\documentclass[10pt]{article}
\usepackage{amsmath}
\usepackage{amssymb} 
\usepackage{epsbox}
\usepackage{amsthm}
\usepackage{graphicx}

 \newtheorem{thm}{Theorem}
 \newtheorem{cor}{Corollary}

 \newtheorem{prop}{Proposition}
 \newtheorem{lem}{Lemma}

 {\theoremstyle{definition}
 \newtheorem{rem}{Remark}}
 {\theoremstyle{definition}
  \newtheorem{defn}{Definition}}
 {\theoremstyle{definition}
 \newtheorem{exam}{Example}}

 \newcommand{\s}{\mathcal{S}}
 \newcommand{\mC}{\mathcal{C}}
\newcommand{\bn}{\mathbf{n}}

\title{Finite orbits of Hurwitz actions on braid systems}
\author{Tetsuya Ito}
\date{}

\begin{document}
  
 \maketitle
\begin{abstract} 
 There are natural actions of the braid group $B_{n}$ on $B_{m}^{n}$, the $n$-fold product of the braid group $B_{m}$, called the Hurwitz action. We first study the roots of centralizers in the braid groups. By using the structure of the roots, we provide a criterion for the Hurwitz orbit to be finite and give an upper bound of the size for a finite orbit in $n=2$ or $m=3$ case. 
\end{abstract}

\section{Introduction}

\footnote[0]{2000 Mathematics Subject Classification: Primary 20F36}
Let $S_{n}$ be the degree $n$ symmetric group and $B_{n}$ be the braid group of $n$-strands, defined by the presentation  
\[
B_{n} = 
\left\langle
\sigma_{1},\sigma_{2},\cdots ,\sigma_{n-1}
\left|
\begin{array}{ll}
\sigma_{i}\sigma_{j}=\sigma_{j}\sigma_{i} & |i-j|\geq 2 \\
\sigma_{i}\sigma_{j}\sigma_{i}=\sigma_{j}\sigma_{i}\sigma_{j} & |i-j|=1 \\
\end{array}
\right.
\right\rangle
.
\]

The pure braid group $P_{n}$ is defined as the kernel of the natural projection $\pi:B_{n} \rightarrow S_{n}$, defined by $\sigma_{i} \mapsto (i,i+1)$. 
For a braid $\beta = \sigma_{i_{1}}^{e_{1}}\sigma_{i_{2}}^{e_{2}}\cdots \in B_{n}$, the exponent sum of $\beta$ is defined by the integer $e_{1}+e_{2}+\cdots$ and denoted by $e(\beta)$.   
 
  {\it A braid system} of {\it degree} $m$ and {\it length} $n$ is, by definition, an element of the $n$-fold product of the braid group $B_{m}$. The {\it Hurwitz action} is an action of $B_{n}$ on the set of length $n$, degree $m$ braid systems $B_{m}^{n}$, defined by
\[(\beta_{1},\beta_{2},\ldots,\beta_{n})\cdot \sigma_{i} = (\beta_{1},\beta_{2},\ldots,\beta_{i-1},\beta_{i+1},\beta_{i}^{\beta_{i+1}},\beta_{i+2},\ldots , \beta_{n})\]
where we denote $\beta_{i+1}^{-1}\beta_{i}\beta_{i+1}$ by $\beta_{i}^{\beta_{i+1}}$.

Diagrammatically, the definition of the Hurwitz action can be understood by the Figure \ref{fig:hurwitzaction}. More generally, we can define the action of the braid group $B_{n}$ on the $n$-fold product of groups or racks in a similar way \cite{br}. 

\begin{figure}[htbp]
 \begin{center}
\includegraphics[width=45mm]{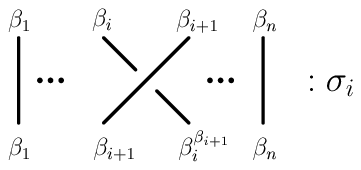}
 \end{center}
 \caption{Diagrammatic description of the Hurwitz action}
 \label{fig:hurwitzaction}
\end{figure}

For a braid system $\s$, we denote the orbit of $\s$ under the Hurwitz action by $\s\cdot B_{n}$ and call it the {\it Hurwitz orbit}. The main object studied in this paper is finite Hurwitz orbit. Although the definition of the Hurwitz action is simple, a computation of a Hurwitz orbit is not so easy. Some interesting calculations for Hurwitz orbits for Artin groups are done in \cite{h}. We study the structure of a finite Hurwitz orbit for general braid systems, and provide an upper bound of finite Hurwitz orbit for length $2$ or degree $3$ braid systems.

To study finite Hurwitz orbit, we first study the roots of centralizers of braids. We denote by $Z(\beta)$ the centralizer of an $n$-braid $\beta$. The following results use the structure theorem of centralizers in \cite{gw}, which is based on the classification of surface automorphisms due to Thurston \cite{flp}.
\begin{thm}
\label{thm:centerroot}
Let $\alpha,\beta \in B_{m}$ and suppose $\alpha \in Z(\beta^{s})$ for some $s>0$.
\begin{enumerate} 
\item If $\beta$ is periodic, then $\alpha \in Z(\beta^{m\cdot(m-1)})$.
\item If $\beta$ is pseudo-Anosov, then $\alpha \in Z(\beta)$.
\item If $\beta$ is reducible, then $\alpha \in Z(\beta^{(m-1)!})$.
\end{enumerate}
\end{thm}

This result is interesting in its own right. This theorem implies, for two $n$-braids $\alpha$ and $\beta$, if $\alpha^{M}$ and $\beta^{M}$ commute for some non-zero integer $M$, then $\alpha^{n!}$ and $\beta^{n!}$ always commute.

Now we return to consider finite Hurwitz orbit. To state our results, we introduce a notion of a reducible braid system. We say a length $n$ braid system $\s=(\beta_{1},\beta_{2},\ldots,\beta_{n})$ is {\it reducible} if there exists a non-trivial partition $I \coprod J$ of the set $\{1,2,\ldots,n\}$ such that  
 $\beta_{i} \beta_{j} = \beta_{j}\beta_{i}$ for all $i\in I, j\in J$. 
  For a reducible braid system $\s$, let us define $\s' = (\beta_{i_{1}},\beta_{i_{2}},\ldots, \beta_{i_{l}})$, where $i_{p} \in I, i_{p}<i_{p+1}$ and $\s'' = (\beta_{j_{1}},\beta_{j_{2}},\ldots,\beta_{j_{m}})$, where $j_{p} \in J, j_{p} < j_{p+1}$.
  
As is easily checked, if a reducible braid system $\s$ has finite Hurwitz orbit, then Hurwitz orbits of $\s'$ and $\s''$ are also finite, and the inequality 
\[ 
\sharp(\s \cdot B_{n}) \leq 
\left( 
\begin{array}{c} 
n \\
l  
\end{array} 
\right) 
\sharp(\s' \cdot B_{l}) \cdot \sharp(\s'' \cdot B_{n-l})
\]
 holds. So in this paper we mainly focus on irreducible braid systems.
Our main results are the following.

\begin{thm}[Finiteness theorem for length $2$ braid systems]
\label{thm:b2finiteness}
 Let $\s$ be a degree $m$, length two braid system having finite Hurwitz orbit.
\begin{enumerate}
\item  If $m=3$, then $\sharp (\s \cdot B_{2}) \leq 6$.
\item  If $m \geq 4$, then $\sharp(\s \cdot B_{2}) \leq 2\cdot (m-1)!$.
\end{enumerate}
\end{thm}
\begin{thm}[Finiteness theorem for degree $3$ braid systems]
\label{thm:3finite}
Let $\s$ be a degree $3$, length $n$ braid system having finite Hurwitz orbit.
\begin{enumerate}
\item If $n=2$, then $\sharp \s \cdot B_{n} \leq 6$.
\item If $n\geq 3$, then $\sharp \s \cdot B_{n} \leq 27 \cdot n!$.
\item If $n\geq 5$, then $\s$ is reducible.
\end{enumerate}
\end{thm}

\section{Roots of centralizers}
\subsection{Structure of the centralizers of braids}
   In this subsection we briefly review the results of \cite{gw}, the structure of the centralizers of a braid. The braid group $B_{n}$ is naturally identified with the relative mapping class group $MCG(D_{n},\partial D_{n})$ of the $n$-punctured disc $D_{n}$, which is the group of isotopy classes of homeomorphisms of $D_{n}$ which fixes $\partial D_{n}$ pointwise \cite{bi}.
 
From the Nielsen-Thurston theory, each element of the braid group $B_{n}$ is classified into the following three types, {\it periodic, reducible}, and {\it pseudo-Anosov} according to its dynamical property. See \cite{flp} for details of Nielsen-Thurston theory. In this paper we treat the trivial element of $B_{n}$ as a periodic braid.

 A periodic braid is a braid some of whose powers belong to the center of the braid group, which is an infinite cyclic group generated by the square of the Garside element
\[ \Delta^{2}=\{(\sigma_{1}\sigma_{2}\cdots\sigma_{n-1})(\sigma_{1}\cdots\sigma_{n-2})\cdots(\sigma_{1}\sigma_{2})(\sigma_{1})\}^{2}.\]

 It is classically known \cite{e} that each periodic $n$-braid is conjugate to either \[ (\sigma_{1}\sigma_{2}\cdots\sigma_{n-1})^{m}\textrm{ or } (\sigma_{1}\sigma_{2}\cdots\sigma_{n-1}\sigma_{1})^{m}\]
 for some integer $m$. This implies that the $n$-th or $(n-1)$-st powers of a periodic braid always belong to the center of $B_{n}$.
 
  The centralizer of a periodic braid is simple in special case. From the above facts, we can write a periodic $n$-braid as 
\[\gamma^{-1}(\sigma_{1}\sigma_{2}\cdots\sigma_{n-1})^{k}\gamma \textrm{ or } \gamma^{-1}(\sigma_{1}\sigma_{2}\cdots\sigma_{n-1}\sigma_{1})^{k}\gamma. \]
 In the former case, if $k$ and $n$ are coprime, then the centralizer $Z(\beta)$ is an infinite cyclic group generated by $\gamma^{-1}(\sigma_{1}\sigma_{2}\cdots\sigma_{n-1})\gamma$.
Similarly, in the latter case, if $k$ and $n-1$ are coprime, then the centralizer $Z(\beta)$ is an infinite cyclic group generated by $\gamma^{-1}(\sigma_{1}\sigma_{2}\cdots\sigma_{n-1}\sigma_{1})\gamma$ \cite[Proposition 3.3]{gw}. If $k$ and $n$ (or $n-1$) are not coprime, then the centralizer of periodic braids are isomorphic to the braid group of annulus \cite[Corollary 3.6]{gw}.

   A pseudo-Anosov braid is a braid which is represented by a pseudo-Anosov homeomorphism. A pseudo-Anosov homeomorphism $f$ is a homeomorphism which has the two invariant measured foliations $(\mathcal{F}^{s},\mu^{s})$, $(\mathcal{F}^{u},\mu^{u})$ called the {\it stable} and {\it unstable foliation} and the real number $\lambda > 1$ called the {\it dilatation}. They satisfy the condition $f(\mathcal{F}^{s},\mu^{s})=(\mathcal{F}^{s},\lambda^{-1} \mu^{s})$ and $f(\mathcal{F}^{u},\mu^{u})=(\mathcal{F}^{u},\lambda\mu^{u})$. 

The centralizer $Z(\beta)$ of a pseudo-Anosov braid $\beta$ is also simple. The centralizer $Z(\beta)$ is isomorphic to the rank two free abelian group generated by one pseudo-Anosov element and one periodic element, both of which preserve the invariant foliations of $\beta$ \cite[Proposition 4.1]{gw}. In particular, all braids in $Z(\beta)$ are irreducible.

   A reducible braid is a braid which preserves a non-empty essential submanifold $\mC$ of $D_{n}$. In this paper we adapt the convention that every reducible braid is non-periodic. By taking an appropriate conjugation, each reducible braid $\beta$ can be converted to the following simple form, called a {\it standard form}. 

Regard $\mC$ as a set of essential circles. A collection of essential circles $\mC$ is called {\it standard curve system} if $\mC$ satisfies the following two conditions.
\begin{enumerate}
\item The center of each circle in $\mC$ lies on $x$-axis.  
\item For any two distinct circles $C$ and $C'$ in $\mC$, $C$ does not enclose $C'$.
\end{enumerate}

By taking an appropriate conjugation, we can always assume that a reducible braid $\beta$ preserves a standard curve system $\mC$. The braid $\beta$ acts on the set $\mC$ as a permutation of circles. Let us denote the orbit decomposition of $\mC$ by $\mC= \mC_{1} \cup \mC_{2}\cup \cdots \cup\mC_{l}$, where $\mC_{i} = \{ C_{i,1}, \ldots,C_{i,r_{i}} \}$. We choose the numbering $C_{i,j}$ so that $\beta(C_{i,j}) = C_{i,j+1}$ (modulo $r_{i}$) holds.

Let us denote the number of punctures in the circle $C_{i,j}$, which is independent of $j$, by $c_{i}$.
Then the orbit decomposition defines the weighted partition $\bn$ of an integer $n$, $\bn: n=c_{1}r_{1}+c_{2}r_{2}+\cdots + c_{k}r_{k}$.

In this situation, we can write the reducible braid $\beta$ as a composition of two parts. The first part is the {\it tubular braid}, which is a braiding of tubes corresponding to the permutation of the circles. Each tube contains some numbers of parallel strands (possibly one) which are not braided inside the tube. The other part is the {\it interior braids} $\beta_{i,j}$, which are braids inside the tube sending the circle $C_{i,j-1}$ to $C_{i,j}$. We denote the braid obtained by regarding each tube of the tubular braid as one strand by $\beta_{ext}$ and call it the {\it exterior braid}. The interior braids $\beta_{i,j}$ and the exterior braid $\beta_{ext}$ are chosen so that they are non-reducible.

Using the above notions, we denote the reducible braid $\beta$ as
\[ \beta= \beta_{ext}(\beta_{1,1}\oplus\beta_{1,2}\oplus\cdots \oplus \beta_{k,r_{k}})_{\bn}\]
and call such a form of the braid the {\it standard form}. See Figure \ref{fig:standard}.

\begin{figure}[htbp]
 \begin{center}
\includegraphics[width=80mm]{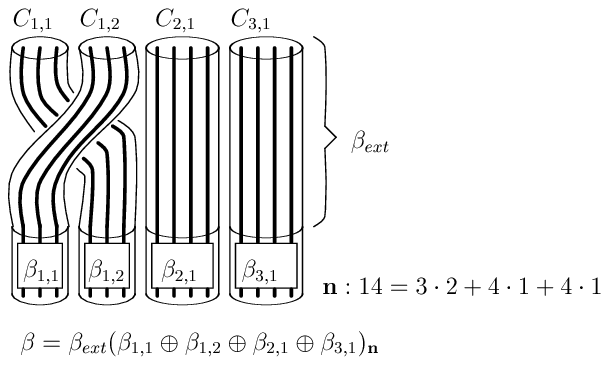}
 \end{center}
 \caption{Standard form of reducible braids}
 \label{fig:standard}
\end{figure}

 We can make a reducible braid in standard form much simpler by taking a further conjugation so that the following hold.
\begin{enumerate}
\item Each interior braid $\beta_{i,j}$ is a trivial braid unless $j=1$.
\item $\beta_{i,1}$ and $\beta_{j,1}$ are non-conjugate unless $\beta_{i,1} = \beta_{j,1}$. 
\end{enumerate}

After this modification, we denote the non-trivial interior braids $\beta_{[i,1]}$ simply by $\beta_{[i]}$.
Now the whole braid $\beta$ is written as
\[
\beta  =  \beta_{ext}\cdot(\beta_{[1]}\oplus \underbrace{ 1 \oplus \cdots \oplus 1}_{r_{1}-1} \oplus \beta_{[2]} \oplus \cdots \oplus \beta_{[k]} \oplus \underbrace{ 1 \oplus \cdots \oplus 1}_{r_{k}-1})_{\bn}. \]
We denote this special standard form of a reducible braid by
 \[ \beta= \beta_{ext}\cdot(\beta_{[1]},\beta_{[2]},\ldots,\beta_{[k]})_{\bn}. \]
and call it the {\it normal form}.

Let $\beta=\beta_{ext}\cdot(\beta_{[1]},\beta_{[2]},\ldots,\beta_{[m]})_{\bn}$ be a normal form of a reducible braid which preserves a standard curve system $\mC$. Then the centralizer of $\beta$ is described as follows.

Every $\alpha \in Z(\beta)$ preserves $\mC$, hence $\alpha$ is written as a standard form. In particular, the exterior part $\alpha_{ext}$ of $\alpha$ also induces the permutation of circles in $\mC$.  
We say $\alpha_{ext}$ is {\it consistent with } $\beta_{ext}$ if 
$\alpha_{ext} (C_{i,k}) = C_{j,l}$ then $\beta_{[i]}=\beta_{[j]}$ holds. Let $Z_{0}(\beta_{ext})$ be a subgroup of $Z(\beta_{ext})$ defined by 
\[ Z_{0}(\beta_{ext}) = \{ \alpha_{ext} \in Z(\beta_{ext}) \: | \: \alpha_{ext} \textrm{ is consistent with } \beta_{ext} \}. \]
 Then $Z(\beta)$ is described by the following split exact sequence \cite[Theorem 1.1]{gw}.
\[ 1 \longrightarrow Z(\beta_{[1]}) \times Z(\beta_{[2]})\times \cdots \times Z(\beta_{[k]}) \stackrel{i}{\longrightarrow} Z(\beta) \stackrel{j}{\longrightarrow} Z_{0}(\beta_{ext}) \longrightarrow 1.\]

The map $i$ is defined by
\[ i (\alpha_{[1]},\alpha_{[2]},\ldots ,\alpha_{[k]}) = 1\cdot(\underbrace{\alpha_{[1]}\oplus \cdots \oplus \alpha_{[1]}}_{r_{1}} \oplus \cdots \oplus \underbrace {\alpha_{[k]} \oplus \cdots \oplus \alpha_{[k]}}_{r_{k}}   )_{\bn}\] 
and the map $j$ is defined by 
\[ j( \alpha_{ext}\cdot (\alpha_{1,1}\oplus \cdots \oplus \alpha_{k,r_{k}} )_{\bn}) = \alpha_{ext}.\]
The splitting $s$ of the above exact sequence is given by 
\[s(\alpha_{ext}) = \alpha_{ext}(1\oplus \cdots \oplus 1)_{\bn} \]
Therefore, for each $\alpha \in Z(\beta)$, we can write $\alpha$ as 
\[ \alpha = \alpha_{ext} \cdot (\alpha_{[1]}^{\oplus r_{1}} \oplus \alpha_{[2]}^{\oplus r_{2}} \oplus \cdots \oplus  \alpha_{[k]}^{\oplus r_{k}}  )_{\bn}.\]

\subsection{Proof of Theorem \ref{thm:centerroot}}

Now we are ready to prove Theorem \ref{thm:centerroot}.

\begin{proof}
The assertion {\it 1.} is immediate because for a periodic braid $\beta \in B_{m}$, $\beta^{m}$ or $\beta^{m-1}$ belongs to $Z(B_{m})$.
The proof of the assertion {\it 2.} is also easy. Suppose $\beta$ is pseudo-Anosov and let $\mathcal{F}^{s},\mathcal{F}^{u}$ and $\lambda$ be the stable, unstable measured foliation and the dilatation of $\beta$.
 Since $\alpha$ belongs to the center of the pseudo-Anosov braid $\beta^{s}$, whose invariant measured foliations are also $\mathcal{F}^{s}$ and $\mathcal{F}^{u}$, $\alpha$ also preserves both $\mathcal{F}^{s}$ and $\mathcal{F}^{u}$. Now the braid $\alpha\beta\alpha^{-1}\beta^{-1}$ preserves the measured foliations $\mathcal{F}^{s}$ and $\mathcal{F}^{u}$ has the dilatation $1$. This implies that the braid $\alpha\beta\alpha^{-1}\beta^{-1}$ is periodic. Since the exponent sum of $\alpha\beta\alpha^{-1}\beta^{-1}$is zero, we conclude that $\alpha\beta\alpha^{-1}\beta^{-1} = 1$. Therefore we obtain $\alpha \in Z(\beta)$.

Now we proceed to the most difficult case, reducible case. 
By taking a conjugate of $\beta$, we may assume $\beta$ is a normal form  
\[ \beta=\beta_{ext}\cdot (\beta_{[1]},\beta_{[2]},\ldots,\beta_{[k]})_{\bn}\]
where $\bn\: : \: m=c_{1}\cdot r_{1} + \cdots + c_{k}\cdot r_{k} $ is an associated weighted partition of $m$. 
Let us define integers $a_{i}$ by $a_{i} = (m-1)! \slash r_{i}$.
Since the exterior part of $\beta^{(m-1)!}$ is a pure braid $\beta_{ext}^{(m-1)!}$, so $\beta^{(m-1)!}$ is written as a normal form  
\[ \beta^{(m-1)!} = \beta_{ext}^{(m-1)!} \cdot ( 
				\underbrace{\beta_{[1]}^{a_{1}},\ldots , \beta_{[1]}^{a_{1}} }_{r_{1}}, 
		\ldots, \underbrace{\beta_{[k]}^{a_{k}}, \ldots , \beta_{[k]}^{a_{k}} }_{r_{k}})_{\bn^{*}}. \]
where $\bn^{*}$ is a weighted partition defined by 
\[ \bn^{*} : m= \underbrace{c_{1}\cdot 1 +\cdots + c_{1} \cdot 1}_{r_{1}} + \cdots + \underbrace{ c_{k} \cdot 1 + \cdots + c_{k} \cdot 1}_{ r_{k}}.\]
		
Let $\alpha \in Z(\beta^{s})$. Then $\alpha \in Z(\beta^{(m-1)! \cdot s})$.
From the normal form of $\beta^{(m-1)! \cdot s}$,
 $\alpha$ can be written as a standard form
\[ \alpha = \alpha_{ext} \cdot (\alpha_{1,1}\oplus \alpha_{1,2}  \oplus \cdots \oplus \alpha_{k,r_{k}})_{\bn^{*}} \]

Since the interior braids $\alpha_{i,j}$ are irreducible $c_{i}$-braid and $\alpha_{i,j} \in Z(\beta_{[i]}^{a_{i}\cdot s})$, from the assertion {\it 1.} and {\it 2.} we obtain $\alpha_{i,j} \in \beta_{[i]}^{c_{i}\cdot(c_{i}-1)}$.

Now observe that $a_{i} \slash c_{i}\cdot(c_{i}-1) = (m-1)! \slash r_{i}c_{i}(c_{1}-1)$ is an integer.
Therefore we conclude that $\alpha_{i,j} \in Z(\beta_{[i]}^{a_{i}})$. By the same argument, we also obtain $\alpha_{ext} \in Z(\beta_{ext}^{(m-1)!})$. 

If $\alpha_{ext}$ is not consistent with $\beta^{(m-1)!}_{ext}$, then there exist pairs $(i,k)$ and $(j,l)$ such that $\alpha_{ext}(C_{[i,k]}) = C_{[j,l]}$ but $(\beta^{(m-1)!})_{i,k} = \beta_{[i]}^{a_{i}} \neq \beta_{[j]}^{a_{j}} = (\beta^{(m-1)!})_{j,l}$ holds.
On the other hand, $\alpha \in Z(\beta^{(m-1)! \cdot s})$ implies that $\alpha_{ext}$ is consistent with $\beta^{(m-1)!\cdot s}_{ext}$. Therefore 
$(\beta^{(m-1)!\cdot s})_{i,k} = \beta_{[i]}^{a_{i}\cdot s} = \beta_{[j]}^{a_{j} \cdot s} = (\beta^{(m-1)! \cdot s})_{j,l}$ holds. 

It is known that the root of a braid is unique up to conjugacy \cite{g}. Therefore the above equality means that $\beta_{[i]}^{a_{i}}$ and $\beta_{[j]}^{a_{j}}$ are conjugate. Since $\beta^{(m-1)!}$ is a normal form, we conclude that $\beta_{[i]}^{a_{i}} = \beta_{[j]}^{a_{j}}$, which is a contradiction.
Thus we conclude that $\alpha_{ext} \in Z_{0}(\beta_{ext}^{(m-1)!})$, so $\alpha \in Z(\beta^{(m-1)!})$.
\end{proof}

We remark that our value $(m-1)!$ for reducible braids case is not optimal. Only the properties of the number $(m-1)!$ we used in the proof is that the number $a_{i} \slash c_{i}\cdot(c_{i}-1) = (m-1)! \slash r_{i}c_{i}(c_{1}-1)$ is an integer and that $\beta_{ext}^{(m-1)!}$ is a pure braid. By considering these two properties more carefully, we can easily decrease our value $(m-1)!$. We give a smallest value for small $m$ for later use.

\begin{prop}
\label{prop:exactbound}
Let $\alpha,\beta \in B_{m}$ and suppose $\alpha \in Z(\beta^{s})$ for some $s > 0$ and $\beta$ is reducible.
\begin{enumerate} 
\item If $m=3$, then $\alpha \in Z(\beta)$ and $Z(\beta)$ is a free abelian group of rank two.
\item If $m=4$, then $\alpha \in Z(\beta^{s})$ for some $s\leq 3$.
\end{enumerate} 
\end{prop}
\begin{proof}
If $m=3$, then we may assume that by taking an appropriate conjugate, the reducible braid $\beta$ can be written by $\beta = \sigma_{1}^{2}(\sigma_{1}^{k} \oplus 1)_{(2,1)}$. Thus the centralizer of $\beta$ is the free abelian group of rank two generated by $\sigma_{1}^{2}(1 \oplus 1)_{(2,1)}$ and $1(\sigma_{1} \oplus 1)_{(2,1)}$. Thus if $\alpha \in Z(\beta^{s})$ for some $s \geq 1$, then $\alpha \in Z(\beta)$ holds.

The proof of $m=4$ case is also a direct calculation of the centralizers. By taking an appropriate conjugation, we may assume that the braid $\beta$ has one of the following forms.
\begin{enumerate}
\item $\beta = \sigma_{1}^{p}(\sigma_{1}^{q}\oplus \sigma_{1}^{r})_{(2,2)}$.
\item $\beta = \sigma_{1}^{2p}(\beta_{int} \oplus 1)_{(3,1)}$ where $\beta_{int} \in B_{3}$.
\item $\beta = \beta_{ext}( \sigma_{1}^{p} \oplus 1 \oplus 1)_{(2,1,1)}$.
\end{enumerate}
In the first case we obtain $\alpha \in Z(\beta^{2})$. In the second and the third case, $\alpha \in Z(\beta^{2})$ or $\alpha \in Z(\beta^{3})$ holds.

\end{proof}

\section{Some computations of Hurwitz actions}

   Now we begin our study of the Hurwitz action. In this section we do some calculations, which will be used later. 
For two braid systems $\s = (\beta_{1},\ldots,\beta_{n})$ and $\s'=(\beta'_{1},\ldots,\beta'_{n})$ having the same degree and length, we say $\s$ and $\s'$ are conjugate if $\beta'_{i}= \alpha^{-1}\beta_{i}\alpha$ for some braid $\alpha$ and all $i=1,2,\ldots,n$.
Then there is a one-to-one correspondence between two Hurwitz orbits $\s \cdot B_{n}$ and $\s' \cdot B_{n}$ if two braid systems $\s$ and $\s'$ are conjugate. So we try to take a conjugate of braid systems so that computations are easier.
  
Since the pure braid group $P_{n}$ has finite index $n!$ in $B_{n}$, to classify the finite orbits of $B_{n}$, it is sufficient to consider the orbits of pure braid group $P_{n}$. 
For $i=1,2,\ldots,n-1$, let $c_{i}$ be the pure braid defined by 
\[ c_{i} = (\sigma_{1}^{-1}\cdots \sigma_{i-1}^{-1}) \sigma_{i}^{2} (\sigma_{i-1}\cdots\sigma_{1}) \]
 and 
$F_{n-1}$ be a subgroup of $P_{n}$ generated by $\{c_{1},c_{2},\ldots,c_{n-1}\}$. It is known that $F_{n-1}$ is a free group of rank $n-1$ and there exists a split exact sequence
\[ 1 \rightarrow F_{n-1} \rightarrow P_{n} \rightarrow P_{n-1} \rightarrow 1.\]
Hence the pure braid group $P_{n}$ can be described as a semi-direct products of free groups,
\[ P_{n}= P_{n-1} \ltimes F_{n-1} = F_{1} \ltimes F_{2} \ltimes \cdots \ltimes F_{n-1}. \]

See \cite{bi} for details.
Thus, to classify or estimate the size of finite Hurwitz orbit, it is sufficient to consider the $F_{n}$ actions.

Now we compute some actions of element of $F_{n}$.
\begin{lem}
\label{lem:calc}
Let $\s = (\beta_{1},\beta_{2},\ldots,\beta_{n})$ be a length $n$ braid system.
\begin{enumerate}
\item For all $k$ and $i$,
\begin{eqnarray*}
\hspace{1cm} \s \cdot c_{i}^{k} & = & (\beta_{1}^{(\beta_{1}\beta_{i+1})^{k}},\beta_{2}^{ (\beta_{i+1}\beta_{1})^{-k}(\beta_{1}\beta_{i+1})^{k}},\ldots, \\
& & \hspace{1cm} \beta_{i}^{(\beta_{i+1}\beta_{1})^{-k}(\beta_{1}\beta_{i+1})^{k}},\beta_{i+1}^{(\beta_{1}\beta_{i+1})^{k}},\beta_{i+2},\ldots, \beta_{n}).
\end{eqnarray*}
\item For $j>2$,
\[
\ \s \cdot (c_{1}c_{2}\cdots c_{j})^{k} = ( \beta_{1}^{ C^{k}}, \beta_{2}^{(\beta_{1}^{-1} C)^{-k} C^{k}}, \ldots,  \beta_{j+1}^{(\beta_{1}^{-1} C)^{-k} C^{k}},\beta_{j+2},\ldots,\beta_{n} ).\]
 where $C=\beta_{1}\beta_{2}\ldots\beta_{j+1}$.

\item Let $\Delta_{(i,j)} = (\sigma_{i}\sigma_{i+1}\cdots\sigma_{j})(\sigma_{i}\sigma_{i+1}\cdots \sigma_{j-1})\cdots(\sigma_{i}\sigma_{i+1})(\sigma_{i})$. Then 
\[ \s\cdot \Delta_{(i,j)}^{2p} = (\beta_{1},\beta_{2},\ldots,\beta_{i-1}, \beta_{i}^{C^{p}},\ldots,\beta_{j}^{C^{p}},\beta_{j+1},\ldots, \beta_{n}). \]
 where $C= \beta_{i}\beta_{i+1}\cdots \beta_{j}$.

\end{enumerate}
\end{lem}
\begin{proof}
Direct computation.
\end{proof}

\section{Partial Coxeter element}

   In this section, we provide a finiteness and infiniteness criterion of Hurwitz orbits for general degree and length by using the notion of (partial) Coxeter element. The partial Coxeter element argument provides a strong restriction for the finiteness of Hurwitz orbit and gives evidence that finite Hurwitz orbits with non-commutative entries are rare.
 
\begin{defn}
For a braid system $\s=(\beta_{1},\ldots,\beta_{n}) \in B_{m}^{n}$ and strictly increasing sequence of integers $I = \{1 \leq i_{1}<i_{2}<\cdots< i_{k} \leq m\}$, we define $C_{I}(\s)$, the {\it partial Coxeter element} of $\s$ by $C_{I}(\s) = \beta_{i_{1}}\beta_{i_{2}}\cdots\beta_{i_{k}}.$
For the sequence $I=\{1,2,3,\ldots,m\}$, we call $C_{I}(\s)$ the {\it (full) Coxeter element} of $\s$ and denote it by $C(\s)$.
\end{defn}

From the definition of the Hurwitz action, the full Coxeter element $C(\s)$ is invariant under the Hurwitz action, so it is an invariant of the Hurwitz orbit. 
On the other hand, the partial Coxeter element $C_{I}(\s)$ might dramatically change by the Hurwitz action. Even the Nielsen-Thurston types might change.
Now Lemma \ref{lem:calc} and the knowledge of the centralizers provide the following criterion of finiteness.

\begin{thm}[Partial Coxeter element criterion]
\label{thm:partialcox}
Let $\s=(\beta_{1},\beta_{2},\ldots,\beta_{n})$ be a braid system of degree $m$, length $n$ having the finite Hurwitz orbit $\s\cdot B_{n}$ 
and $I = \{1 \leq i_{1}<i_{2}<\cdots< i_{k} \leq n\}$ be a strictly increasing sequence of integers of length $k \geq 2$.
\begin{enumerate}
\item If $C_{I}(\s)$ is pseudo-Anosov, then $\beta_{i_{1}},\beta_{i_{2}},\ldots,\beta_{i_{k}}$ are irreducible and commutative.
\item If $C_{I}(\s)$ is reducible, then $\beta_{i_{1}},\beta_{i_{2}},\ldots,\beta_{i_{k}}$ preserves the same essential 1-submanifold. Especially, they are not pseudo-Anosov.
\item If $C_{\{1,2,\ldots,j\}}$ is periodic, then $\s\cdot (c_{1}c_{2}\cdots c_{j-1})^{r} = \s$ for some $1 \leq r \leq m! $.
\end{enumerate}
\end{thm}
\begin{proof}
First we prove {\it 1.} and {\it 2.}
By considering the action of an appropriate braid, there is a braid system $\s'$ in the Hurwitz orbit of $\s$, which is written as $\s' = (\beta_{i_{1}},\beta_{i_{2}},\ldots, \beta_{i_{k}}, \beta'_{k+1},\ldots)$.
From Lemma \ref{lem:calc} (3), 
\[ \s' \cdot \Delta_{(1,k)}^{2p} =  ( \beta_{i_{1}}^{c^{p}},\beta_{i_{2}}^{c^{p}},\ldots, \beta_{i_{k}}^{c^{p}}, \beta'_{k+1},\ldots)\] 
 where $c=C_{I}(\s)$.
Since $\s \cdot B_{n}$ is finite, $\beta_{i_{j}} \in Z(c^{p})$ for some $p>0$. This means all of  $\beta_{i_{j}}$ are irreducible and commutative if $c$ is pseudo-Anosov, and all of $\beta_{i_{j}}$ preserve the same $1$-submanifold if $c$ is reducible.

 Next we prove {\it 3.}   
Let $C= \beta_{1}\beta_{2}\cdots\beta_{j}$ be the partial Coxeter element and $q$ be a period of $C$.
From Lemma \ref{lem:calc} {´it 2.},
\[ \s\cdot (c_{1}c_{2}\cdots c_{j-1})^{p} = (\beta_{1}^{ C^{p} },
 \beta_{2}^{(\beta_{1}^{-1}C)^{-p}C^{p}}, \ldots , \beta_{j}^{(\beta_{1}^{-1}C)^{-p}C^{p}},\beta_{j+1},\ldots, \beta_{n}). \]
Since $\s \cdot B_{n}$ is finite, we can find $0 < p$ satisfying $\s\cdot (c_{1}c_{2}\cdots c_{j-1})^{p}=\s$. Then 
\begin{eqnarray*}
\s\cdot (c_{1}c_{2}\cdots c_{j-1})^{pq}
 & = &  (\beta_{1},\beta_{2}^{(\beta_{1}^{-1}C)^{-pq}}, \ldots, \beta_{j}^{(\beta_{1}^{-1}C)^{-pq}},\beta_{j+1},\ldots,\beta_{n} ) \\
 & = &  (\beta_{1},\beta_{2},\beta_{3},\ldots,\beta_{n}).
\end{eqnarray*} 
Thus all of $\beta_{2},\beta_{3},\ldots,\beta_{j}$ belong to the centralizer of $(\beta_{1}^{-1}C)^{pq}$.
From Theorem \ref{thm:centerroot}, there exists $s \leq m!$ such that all of $\beta_{2},\beta_{3},\cdots,\beta_{j} \in Z( (\beta_{1}C^{-1})^{s})$.
Therefore, we conclude that $\s\cdot (c_{1}c_{2}\cdots c_{j-1})^{r} = \s$ for some $0<r \leq m!$.
\end{proof}

 These result imply that each entry of a braid system with finite Hurwitz orbit must satisfy the following conditions.
 \begin{itemize}
 \item If its full Coxeter element is pseudo-Anosov, then all of its entries must be commutative.
 \item If its full Coxeter element is reducible, then all of its entries must not be pseudo-Anosov and preserve the same 1-submanifold $\mC$.
 \end{itemize}

 Using this condition, sometimes we can easily check whether the Hurwitz orbit is finite or not.

\begin{exam}

Now we give some examples.

\begin{enumerate}
\item Let $\s =(\sigma_{1},\sigma_{2}^{2},\sigma_{1})$. Each entry of $\s$ is reducible and the full Coxeter element is also reducible. However, $\sigma_{1}$ and $\sigma_{2}$ do not preserve the same essential $1$-submanifolds, so we conclude that $\s$ has infinite Hurwitz orbit.  
\item Let $\s = (\sigma_{1},\sigma_{1},\sigma_{1},\sigma_{1},\sigma_{2})$. It is easily checked that braid systems $(\sigma_{1},\sigma_{2})$, $(\sigma_{1},\sigma_{1},\sigma_{2})$ and $(\sigma_{1},\sigma_{1},\sigma_{1},\sigma_{2})$ have finite Hurwitz orbits. However, the Hurwitz orbit of $\s$ is infinite because the full Coxeter element is pseudo-Anosov but $\sigma_{1}$ is reducible.
\end{enumerate}
\end{exam}

As these examples suggest, a braid system might have infinite Hurwitz orbit even if its entries have simple relations.

\section{Classification of finite Hurwitz orbits}

   Now we begin a classification of finite Hurwitz orbits. 
   
\subsection{Length two braid systems}

First of all, we prove Theorem \ref{thm:b2finiteness}. 
  
\begin{proof}[Proof of Theorem \ref{thm:b2finiteness}]
From Lemma \ref{lem:calc} {\it 1.}, $ (\beta_{1},\beta_{2})\sigma_{1}^{2p} =(\beta_{1}^{(\beta_{1}\beta_{2})^{p}}, \beta_{2}^{(\beta_{1}\beta_{2})^{p}})$ holds. Since the Hurwitz orbit of $\s$ is finite, $\beta_{1},\beta_{2} \in Z((\beta_{1}\beta_{2})^{p})$ for some $p>0$. From Theorem \ref{thm:centerroot}, $ p \leq \max\{(m-1)!, m\}$, so the conclusion holds.
\end{proof}

 As in the remark after Theorem \ref{thm:centerroot}, this upper bound is not sharp for general $m$.
For $m=3,4$, we give an accurate upper bound.

\begin{cor}
Let $\s$ be a degree $m$, length $2$ braid system having finite Hurwitz orbit. 
\begin{enumerate}
\item If $m=3$, $\sharp( \s \cdot B_{2} )\leq 6$. 
\item If $m=4$, $\sharp( \s \cdot B_{2}) \leq 8$. 
\end{enumerate}
\end{cor}

The above upper bounds are exact. $\sharp ((\sigma_{1}^{-1},\sigma_{1}^{2}\sigma_{2}) \cdot B_{2}) = 6$ and $\sharp ((\sigma_{1},\sigma_{2}\sigma_{3}) \cdot B_{2}) = 8$.
We remark that there is no universal bound for $\sharp (\beta_{1},\beta_{2})\cdot B_{2}$ if we do not fix the degree $m$. For $m \geq 4$, the size of the Hurwitz orbit of the braid system $(\sigma_{1},\sigma_{2}\sigma_{3}\cdots\sigma_{m-1})$ is $2m$.

\subsection{Normal form of periodic $3$-braids}

Next we study degree $3$ braid systems, where difficulties due to the fact $B_{3}$ is not abelian arise.

Recall that the centralizer of a $3$-braid $\beta$ is abelian unless $\beta$ is central in $B_{3}$. Our classification result relies on this special feature of $B_{3}$.
 In this subsection, we briefly summarize the dual Garside structure of $B_{3}$ and the left normal forms and prepare some lemmas which will be used. See \cite{bkl} for details. 
 
 Let $a_{1,2}= \sigma_{1}$, $a_{2,3}=\sigma_{2}$, $a_{1,3}= \sigma_{2}^{-1}\sigma_{1}\sigma_{2}$ and $\delta=a_{1,2}a_{2,3}=a_{2,3}a_{1,3}=a_{1,3}a_{1,2}$.
Using the braids $\{a_{1,2},a_{2,3},a_{1,3}\}$, the braid group $B_{3}$ is presented by 
\[ B_{3} = \langle a_{1,2},a_{2,3},a_{1,3} \:|\: a_{1,2}a_{2,3}=a_{2,3}a_{1,3}=a_{1,3}a_{1,2} \rangle\] 

 Each 3-braid $\beta \in B_{3}$ has the one of the following unique word representative $N(\beta)$, called the (left-greedy) normal form. 
\[ N(\beta) = \left \{
\begin{array}{l} 
\delta^{m}a_{1,2}^{p_{1}}a_{1,3}^{p_{2}}a_{2,3}^{p_{3}}a_{1,2}^{p_{4}}\cdots a_{*,*}^{p_{k}}\\
\delta^{m}a_{1,3}^{p_{1}}a_{2,3}^{p_{2}}a_{1,2}^{p_{3}}a_{1,3}^{p_{4}}\cdots a_{*,*}^{p_{k}}\\
\delta^{m}a_{2,3}^{p_{1}}a_{1,2}^{p_{2}}a_{1,3}^{p_{3}}a_{2,3}^{p_{4}}\cdots a_{*,*}^{p_{k}}\\
\end{array}
\right.
 \]
 where $p_{i}$ is a positive integer.
In the normal form, the integer $m$ is called the {\it supremum} of $\beta$ and denoted by $\sup(\beta)$. 
We define $d(\beta)$, the {\it depth} of $\beta$, by $d(\beta)=k$.

\begin{lem}
\label{lem:per}
For a periodic $3$-braid $\beta$, if $d(\beta) \neq 0$, $d(\beta) + \sup(\beta) \equiv 2\: (mod\:3)$. 
\end{lem}
\begin{proof}

Let $\beta$ be a periodic $3$-braid and $s=\sup(\beta)$, $d=d(\beta)$.
We only prove $s \equiv 0 \: (mod\:3) $ case. Other cases are similar.
Assume that $d \not \equiv 2 \: (mod\:3)$. 
Then by taking a conjugation by $\delta$, we can assume that the normal form of $\beta$ is either
 \[ N(\beta) = 
\left\{
\begin{array}{l}
\delta^{3s'}a_{1,2}^{p_{1}}\cdots a_{2,3}^{p_{d}} \textrm{  or} \\
\delta^{3s'}a_{1,2}^{p_{1}}\cdots a_{1,2}^{p_{d}}. 
\end{array}
\right.
\]

In either case, the normal form of $\beta^{6}$ is given by
 \[ N(\beta^{6}) =
\left\{ 
\begin{array}{l}
\delta^{18s'}(a_{1,2}^{p_{1}}\cdots a_{2,3}^{p_{d}})(a_{1,2}^{p_{1}}\cdots a_{2,3}^{p_{d}})\cdots (a_{1,2}^{p_{1}}\cdots a_{2,3}^{p_{d}}) \textrm{ or} \\
\delta^{18s'}(a_{1,2}^{p_{1}}\cdots a_{1,2}^{p_{d}+p_{1}})(a_{1,3}^{p_{2}}\cdots a_{2,3}^{p_{d}+p_{1}})\cdots(a_{1,3}^{p_{2}}\cdots a_{2,3}^{p_{d}}). 
\end{array}
\right.
\]
Therefore, $\beta$ is not periodic.
\end{proof}

Now we prove the key lemma which plays an important role in proving our finiteness results for degree $3$ braid systems.

\begin{lem}
\label{lem:key}
Let $\alpha$ be a periodic $3$-braid whose period is $3$.
Then for $\beta,\gamma \in B_{3}$, not all of $\beta\gamma,\beta^{\alpha}\gamma,\beta^{\alpha^{2}}\gamma$ are periodic unless either $\beta$ or $\gamma$ belongs to $Z(\alpha)$.
\end{lem}
\begin{proof}
Assume that both $\beta$ and $\gamma$ do not belong to $Z(\alpha)$.
By considering a conjugate of the braid system,
we can assume that $\alpha = \delta^{p}$ and the normal form of $\gamma$ is written as 
\[ N(\gamma) = \delta^{g} a_{1,2}^{p_{1}}a_{1,3}^{p_{2}}\cdots a_{*,*}^{p_{k}}. \]
Since both $\gamma$ and $\beta$ do not commute with $\alpha=\delta^{p}$, we obtain $d(\gamma) \neq 0 $ and $d(\beta) \neq 0$. Now let us denote the normal form of $\beta$ by
\[ N(\beta) = \delta^{b} \cdots a_{i,j}^{q}.\] 
Then for some distinct $e,f \in \{1,2,3\}$, the normal forms of $\beta^{\delta^{e+g}}$ and $\beta^{\delta^{f+g}}$ are given by
\[ 
N(\beta^{\delta^{e+g}}) = \delta^{b} \cdots a_{1,2}^{q}, \;\;
N(\beta^{\delta^{f+g}}) = \delta^{b} \cdots a_{2,3}^{q}. 
\]
 Now the normal forms of $\beta^{\delta^{e}} \gamma$ and $\beta^{\delta^{f}} \gamma$ are written as 
\[ 
N(\beta^{\delta^{e}} \gamma) = \delta^{b+g} \cdots a_{1,2}^{q+p_{1}} a_{1,3}^{p_{2}} \cdots a_{*,*}^{p_{k}}, \;\;
N(\beta^{\delta^{f}} \gamma) = \delta^{b+g} \cdots a_{2,3}^{q}a_{1,2}^{p_{1}} a_{1,3}^{p_{2}} \cdots a_{*,*}^{p_{k}} .
 \]

Thus, $\sup(\beta^{\delta^{e}} \gamma) + d(\beta^{\delta^{e}} \gamma) = b+g+q+k-1$ and $\sup(\beta^{\delta^{f}} \gamma) + d(\beta^{\delta^{f}} \gamma) = b+g+q+k $. By Lemma \ref{lem:per}, we conclude that not both of $\beta^{\delta^{e}} \gamma$ and $\beta^{\delta^{f}} \gamma$ are periodic.
\end{proof}

\subsection{Exponent sum restriction}

In this subsection, we study the exponent sum of the entries of braid systems having finite Hurwitz orbit. 
We observe the following simple, but crucial lemma about degree $3$ braid systems having finite Hurwitz orbits.

\begin{lem}
\label{lem:restriction}
Let $\s = (\beta_{1},\cdots,\beta_{l})$ be a degree $3$ braid system having finite Hurwitz orbit and assume that all of $\beta_{i}$ are not central in $B_{3}$.
 If $e(\beta_{i_{1}})+e(\beta_{i_{2}}) + \cdots + e(\beta_{i_{k}}) \not \equiv \pm 2, 3\;\;(mod\: 6)$ for some $1\leq i_{1} < i_{2} < \cdots < i_{k} \leq l $ $(1 < k <l)$, then all of its entry $\beta_{i}$ are mutually commutative.
\end{lem}

\begin{proof}
With no loss of generality, we can assume that $e(\beta_{1})+e(\beta_{2})+\cdots + e(\beta_{k}) \not \equiv \pm 2, 3 \;\;(mod \: 6)$. 
First we show that $\beta_{k}$ commutes with $\beta_{k+1}$.
Let $C= \beta_{1}\beta_{2}\cdots \beta_{k-1}$.

Using the result of Eilenberg \cite{e} alluded to above and the hypothesis on the exponent sum, the partial Coxeter element $\beta_{1}\beta_{2}\cdots\beta_{k} = C \beta_{k}$ is non-periodic or central in $B_{3}$. Thus $\beta_{k}$ belongs to $Z(C)$. Similarly, by considering the partial Coxeter element of $\s \cdot \sigma_{k}^{2}$ we obtain that $\beta_{k}^{\beta_{k+1}}$ also belongs to $Z(C)$.

First we consider the case $C$ is periodic. Since we have assumed that $\beta_{k}$ is non-central, so $C$ is also non-central. This implies $Z(C)$ is an infinite cyclic group generated by an element having non-zero exponent sum. Thus we conclude $\beta_{k}=\beta_{k}^{\beta_{k+1}}$, so $\beta_{k}$ and $\beta_{k+1}$ commute.

If $C$ is pseudo-Anosov, then $\beta_{k}^{-1}\beta_{k}^{\beta_{k+1}}$ has the dilatation $1$ and zero exponent sum, hence $\beta_{k}^{-1}\beta_{k}^{\beta_{k+1}}=1$.

Finally, if $C$ is reducible, then $\beta_{k+1}$ and $C$ preserve the same essential submanifold because $\beta_{k}$ and $C$ preserve the same essential submanifold.
In $B_{3}$, this implies that $\beta_{k+1}$ also belongs to $Z(C)$. Thus, $\beta_{k}$ and $\beta_{k+1}$ commute. 

For each $i<k<j$, there exists a braid $\alpha \in B_{k}\times B_{n-k} \subset B_{n}$ such that $\s \cdot \alpha = (\beta'_{1},\ldots,\beta'_{k-1},\beta_{i},\beta_{j},\ldots,\beta'_{l} )$. so from the above argument, $\beta_{i}$ commutes with $\beta_{j}$. Therefore all entries of $\s$ commute. 

\end{proof}

This lemma imposes a strong restriction on the exponent sums (modulo 6) for non-commutative braid systems having finite Hurwitz orbit.

\begin{prop}
There are no irreducible braid systems with degree $3$, length $\geq 5$ having finite Hurwitz orbit.
\end{prop}
\begin{proof}
For a braid system $\s =(\beta_{1},\beta_{2},\cdots, \beta_{l})$, having the length $l \geq 5$, we can always find a sequence of integers $1 \leq i_{1} < i_{2} < \cdots < i_{k} \leq l$ $(1<k \leq l)$ such that $e(\beta_{i_{1}}) + \cdots + e(\beta_{i_{k}}) \neq \pm 2,3\;\;(mod\: 6)$. By Lemma \ref{lem:restriction}, this implies all entries of $\s$ commute, so $\s$ is reducible.
\end{proof}

This proves Theorem \ref{thm:3finite} {\it 3}.

\subsection{Degree $3$, length $3$ braid system}

Let $\s=(\beta_{1},\beta_{2},\beta_{3})$ be a length $3$, degree $3$ irreducible braid system having finite Hurwitz orbit. We denote the full Coxeter element $\beta_{1}\beta_{2}\beta_{3}$ by $C$.
As is described in Section 3, we consider the action of the rank two free group $F=F_{2}$ generated by $c_{1}= \sigma_{1}^{2}$ and $c_{2}= \sigma_{1}^{-1}\sigma_{2}^{2}\sigma_{1}$.

To treat degree $3$ braid systems, it is convenient to consider the quotient group $B_{3}' = B_{3}\slash \langle \Delta^{2} \rangle$ because the centralizer $Z(\beta)$ of a non-trivial element $[\beta] \in B_{3}'$ is a cyclic group.
For $\alpha, \beta \in B_{3}$, we denote by $\alpha \equiv \beta$ if $\alpha$ and $\beta $ defines the same elements in $B_{3}'$.

\subsubsection{Orbit graphs}

The Hurwitz orbit $\s \cdot F$ is described by an oriented, labeled graph $G$, which we call the {\it orbit graph} of $\s$.
The set of vertices of $G$ consists of the set of orbits $\s \cdot F$. Two vertices $\s$ and $\s'$ are connected by an edge oriented from $\s$ to $\s'$ labeled by $1$ (resp. $2$) if $\s \cdot c_{1} = \s'$ (resp. $\s \cdot c_{2} = \s'$). 
We will classify the orbit graphs of irreducible braid systems of the degree $3$ and the length $3$. 

A {\it simple vertex} of $G$ is defined as a vertex $\s$ such that $\s \cdot c_{i} = \s $ holds for some $i=1,2$. An {\it i-path} is an edge path of $G$ having the same label $i$ $(i=1,2)$. An {\it alternate path} is an edge-path whose labels alternate. We call a closed $i$-path of length $3$ a {\it triangle}. A triangle is {\it special} if all vertices of the triangle are non-simple.

First of all, we study the fundamental properties of orbit graphs.
\begin{lem}
\label{lem:forbidden}
Let $\s = (\beta_{1},\beta_{2},\beta_{3})$ be an irreducible braid system having finite Hurwitz orbit. Then the orbit graph $G$ of $\s$ has the following properties.

\begin{enumerate}
\item Every closed $i$-path in $G$ has the length at most $3$, and the length $2$ closed $i$-path and length $3$ closed $i$-path does not occur simultaneously. 
\item Every alternate path of length $12$ must be a loop.
\item There exist no subgraphs of the form $(F1) - (F4)$.
\end{enumerate}
\end{lem}

\begin{proof}
The assertion {\it 1.} follows from Proposition \ref{prop:exactbound},
 and the assertion {\it 2.} follows from Theorem \ref{thm:partialcox} {\it 3.}
If there exists a subgraph of the form $(F1)$, then there exists a vertex $\s'= (\beta_{1},\beta_{2},\beta_{3})$ such that $\s' \cdot (c_{1}c_{2}) = \s'$ holds. However this implies $\beta_{1},\beta_{2}$ and $\beta_{3}$ commute, hence it contradicts the assumption that $\s'$ is irreducible. The non-existence of the other subgraphs $(F2)$, $(F3)$ and $(F4)$ are proved by the similar way.
\end{proof}
\begin{figure}[htbp]
 \begin{center}
\includegraphics[width=70mm]{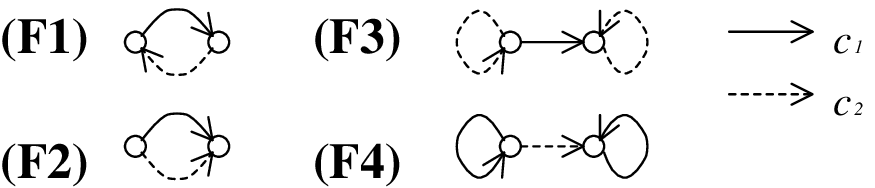}
 \end{center}
 \caption{Forbidden graphs}
 \label{fig:graph}
\end{figure}
We remark that the orbit graph $G$ has a closed $1$-path of the length $2$ (resp. of the length $3$) only if $e(\beta_{1})+e(\beta_{2}) \equiv 3 \:(mod\; 6)$ (resp. $e(\beta_{1})+e(\beta_{2}) \equiv \pm 2 \:(mod\; 6)$). Similarly, $G$ has a closed $2$-path of the length $2$ (resp. of the length $3$) only if $e(\beta_{1})+e(\beta_{3}) \equiv 3\:(mod\; 6)$ (resp. $e(\beta_{1})+e(\beta_{3}) \equiv \pm 2\:(mod\; 6)$).

To extract further restrictions of the orbit graph, we consider the exponent sums.
For an irreducible braid system having finite Hurwitz orbit, from Lemma \ref{lem:restriction}, all possibilities of the exponent sum modulo $6$ are the following. 

\[
(e(\beta_{1}),e(\beta_{2}),e(\beta_{3})) \equiv
\left\{
\begin{array}{l} 
(\pm 2, \pm 1, \pm 1)\hspace{1cm} \cdots (a)\\
(\pm 1, \pm 2, \pm 1)\hspace{1cm} \cdots (b)\\
(\pm 1, \pm 1, \pm 2)\hspace{1cm} \cdots (c)\\
(0, \pm 2, \pm 2), (\pm2, 0, \pm 2), \\
\;\;\;\;(\pm2, \pm 2, 0), (\pm1, \pm 1, \pm 1) \hspace{0.5cm} \cdots(d) 
\end{array}
\right.
\]

We call a braid system whose exponent sum is a pattern $(a)$ a {\it (2,2)-periodic system}.
Similarly, we call a braid system whose exponent sum is a pattern $(b)$, $(c)$ and $(d)$, {\it (2,3)-periodic system}, {\it (3,2)-periodic system}, and  {\it (3,3)-periodic system} respectively.
Now we study each case separately.

\subsubsection{$(2,2)$-periodic systems}

\begin{lem}
\label{lem:22per}
Let $\s$ be a $(2,2)$-periodic system having finite Hurwitz orbit. Then $\s' \cdot (c_{1}c_{2})^{3} = \s'$ holds for all $\s' \in \s \cdot F$. That is, every alternate path of length $6$ must be a loop.
\end{lem}
\begin{proof}
Let $\s'=(\beta_{1},\beta_{2},\beta_{3})$. Then its Coxeter element is a periodic braid with period $3$ and $\beta_{2},\beta_{3} \in Z((\beta_{2}\beta_{3})^{3})$. Therefore by Lemma \ref{lem:calc} $\s \cdot (c_{1}c_{2})^{3} = \s$.
\end{proof}

\begin{prop}
For a $(2,2)$-periodic system $\s$ having finite Hurwitz orbit, the orbit graph $G$ is either $(A)$ or $(B)$ in the Figure \ref{fig:22graph}. Both $(A)$ and $(B)$ are realized as the orbit graph of a braid system.
\end{prop}
\begin{proof}
The orbit graphs of $(2,2)$-periodic systems have no triangles.
By Lemma \ref{lem:forbidden} and \ref{lem:22per}, if there are simple vertices in $G$, we obtain the graph $(A)$. 
Similarly, if there are no simple vertices in $G$, then by Lemma \ref{lem:forbidden} and \ref{lem:22per}, we obtain the graph $(B)$.
The graph $(A)$ appears as the orbit graph of the braid system $(\sigma_{1}^{2},\sigma_{1},\sigma_{2})$, and the graph $(B)$ appears as the orbit graph of the braid system $(\sigma_{1}\sigma_{2},\sigma_{1},\sigma_{2})$.
\end{proof}
\begin{figure}[htbp]
 \begin{center}
\includegraphics[width=70mm]{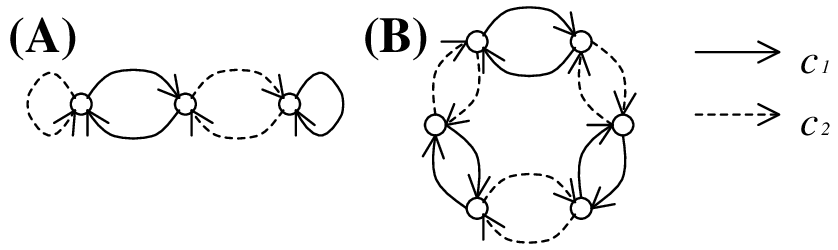}
 \end{center}
 \caption{Orbit graphs of $(2,2)$-periodic systems}
 \label{fig:22graph}
\end{figure}

\subsubsection{$(2,3)$- and $(3,2)$- periodic systems}

Next we consider $(2,3)$- and $(3,2)$-periodic systems.
For simplicity, we consider $(2,3)$-periodic systems. The orbit graphs of $(3,2)$-periodic systems are the same except that the role of $c_{1}$ and $c_{2}$ are interchanged. 

\begin{lem}
\label{lem:23res}
Let $\s=(\beta_{1},\beta_{2},\beta_{3})$ be an irreducible $(2,3)$-periodic system having finite Hurwitz orbit. Then
\begin{enumerate}
\item $\s \cdot (c_{1}c_{2})^{2} \neq \s$.
\item $\s \cdot (c_{1}c_{2})^{2}$ is a simple vertex if and only if $\s$ is a simple vertex.
\end{enumerate} 
\end{lem}
\begin{proof}
Since $\s'$ is a $(2,3)$-periodic system, its Coxeter element $C$ is periodic with period $3$ and $\beta_{2},\beta_{3} \in Z((\beta_{2}\beta_{3})^{2})$.
Thus, $\s \cdot (c_{1}c_{2})^{2} = (\beta_{1}^{C^{2}},\beta_{2}^{C^{2}},\beta_{3}^{C^{2}})$. So $\s \cdot (c_{1}c_{2})^{2}$ is a simple vertex if and only if $\s$ is a simple vertex. If $\s = \s \cdot (c_{1}c_{2})^{2}$, then $\beta_{1}$, $\beta_{2}$ and $\beta_{3}$ commute, hence it contradicts the assumption that $\s$ is irreducible. 
\end{proof}

\begin{prop}
\label{prop:23-graph}
If $\s$ is a $(2,3)$-periodic system having finite Hurwitz orbit, then the orbit graph $G$ is either $(C)$ or $(D)$ in Figure \ref{fig:23graph}. Both $(C)$ and $(D)$ are realized as an orbit graph.
\end{prop}
\begin{proof}
First we consider the case that $G$ has a special triangle.
Let $(\beta_{1},\beta_{2},\beta_{3})$ be a vertex of a special triangle. Then, $\beta_{2}\beta_{1}$, $\beta_{2}^{(\beta_{1}^{-1} \beta_{3}^{-1})}\beta_{1}$ and $\beta_{2}^{ (\beta_{1}^{-1} \beta_{3}^{-1})^{2}}\beta_{1}$ are periodic.
From Lemma \ref{lem:key}, this implies that either $\beta_{2}$ or $\beta_{1}$ belongs to $Z(\beta_{3}\beta_{1})$. Since $\beta_{1}$ and $\beta_{3}$ do not commute, we conclude that $\beta_{2}$ belongs to $Z(\beta_{3}\beta_{1})$. 

Since $\beta_{3}\beta_{1}$ is a periodic braid with period $3$, by taking an conjugation of the braid system, we may assume that $\beta_{2} \equiv \beta_{3}\beta_{1} \equiv \delta^{\pm 1}$. Then the orbit graph of the braid system $(\beta_{1},\delta^{\pm 1},\delta^{\pm 1} \beta_{1}^{-1})$ is the graph $(C)$. The graph $(C)$ is realized as the orbit graph of the braid system $(\sigma_{2},\sigma_{1}\sigma_{2},\sigma_{1})$.
 
 Next we assume that $G$ has no special triangles. Then by Lemma \ref{lem:forbidden} and \ref{lem:23res}, the graph must be the form $(D)$. The graph $(D)$ is realized as the orbit graph of the braid system $(\sigma_{1},\sigma_{1}^{2},\sigma_{2})$.
\end{proof}
\begin{figure}[htbp]
 \begin{center}
\includegraphics[width=80mm]{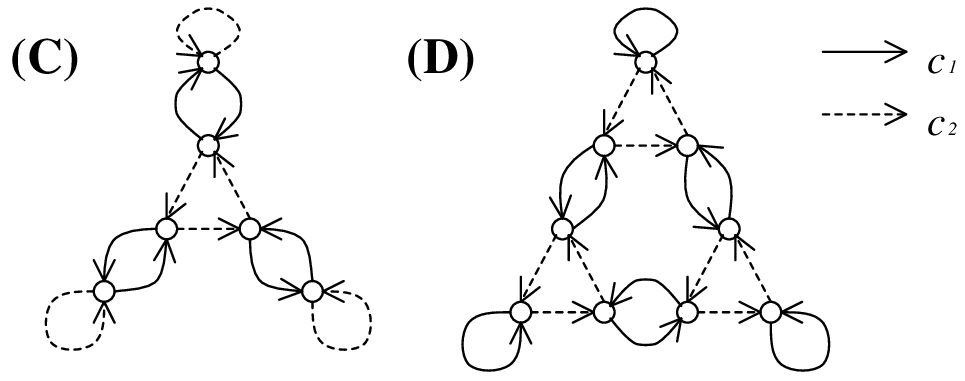}
 \end{center}
 \caption{Orbit graphs of $(2,3)$-periodic systems}
 \label{fig:23graph}
\end{figure}

\subsubsection{$(3,3)$-periodic systems} 

Finally, we consider the orbit graph of $(3,3)$-periodic systems.

\begin{lem}
\label{lem:33res}
Let $\s = (\beta_{1},\beta_{2},\beta_{3})$ be a $(3,3)$-periodic system having finite Hurwitz orbit.
\begin{enumerate}
\item $(e(\beta_{1}),e(\beta_{2}),e(\beta_{3})) \equiv (\pm 1,\pm1,\pm1) \;( mod \;6)$. 
\item $\s \cdot (c_{1}c_{2})^{3}$ is a simple vertex if and only if $\s$ is a simple vertex.
\end{enumerate}
\end{lem}
\begin{proof}

Assume that the exponent sum satisfies 
\[ (e(\beta_{1}), e(\beta_{2}),e(\beta_{3})) \equiv (0,\pm 2,\pm 2) , ( \pm 2, 0,\pm 2), ( \pm 2, \pm 2, 0) \;(mod\;6)\] 
Then, the Coxeter element $C$ of $\s$ is periodic with period $3$, and $\beta_{2},\beta_{3} \in Z((\beta_{2}\beta_{3})^{3})$.
So by Lemma \ref{lem:calc}, $\s \cdot (c_{1}c_{2})^{3} = \s$ holds.

First of all, we show that the orbit graphs of such $(3,3)$-periodic systems have no special triangles.
Assume that there exists a special triangle labeled by $2$. Let $\s = (\beta_{1}, \beta_{2}, \beta_{3})$ be a vertex of a special triangle. Then as in the proof of Proposition \ref{prop:23-graph}, we may assume that $\s = (\beta_{1}, \delta^{\pm 1}, \delta^{\pm 1} \beta_{1}^{-1})$ by taking a conjugation of the braid system. 
Let $T$ be a triangle formed by the vertices $\s$, $\s \cdot c_{1}$ and $\s \cdot c_{1}^{2}$. 
Suppose that $T$ is special. Then, $\beta_{1}(\delta^{\pm 1}\beta_{1}^{-1}) $, $\beta_{1}^{(\beta_{1}\delta^{\pm 1})}(\delta^{\pm 1}\beta_{1}^{-1})$ and $\beta_{1}^{(\beta_{1}\delta^{\pm 1})^{2}}(\delta^{\pm 1}\beta_{1}^{-1})$ are periodic, so
by Lemma \ref{lem:key}, $\beta_{1}$ or $\delta^{\pm 1}\beta_{1}^{-1}$ commutes with $\beta_{1}\delta^{\pm 1}$. This implies $\beta$ and $\delta$ commute.
If $T$ is non-special, then $\beta_{1}^{(\beta_{1}\delta^{\pm 1})}$ or $\beta_{1}^{(\beta_{1}\delta^{\pm 1})^{2}}$ commutes with $\delta^{\pm 1}\beta_{1}^{-1}$. Using the fact that $\beta_{1} \delta^{\pm 1}$ is a periodic braid with period $3$, in either case, we obtain that $\beta_{1}$ commutes with $\delta$.
This contradicts the assumption that $\s$ is irreducible. 
The non-existence of special triangles labeled by $1$ is similar.

Then it is impossible to construct an orbit graph $G$ which satisfies all required properties
\begin{enumerate}
\item $G$ satisfies the condition in Lemma \ref{lem:forbidden}. In particular, all closed $i$-paths in $G$ have the length $3$ or $1$ ($i=1,2$).
\item $G$ has no special triangles.
\item $\s \cdot (c_{1}c_{2})^{3} = \s$ holds for all vertex $\s$ in $G$.
\end{enumerate}
So irreducible braid systems having such exponent sums cannot have finite Hurwitz orbit. This proves {\it 1}.

Now, the Coxeter element $C$ of $\s$ is periodic with period $2$ and $e(\beta_{2})+e(\beta_{3}) \equiv \pm 2 \;(mod \;6)$. Thus, $\s \cdot (c_{1}c_{2})^{3} = (\beta_{1}^{C},\beta_{2}^{C},\beta_{3}^{C})$ holds. So $\s$ is a simple vertex if and only if $\s \cdot (c_{1}c_{2})^{3}$ is a simple vertex.

\end{proof}

\begin{prop}
\label{prop:33-graph}
If $\s$ is a $(3,3)$-periodic system having finite Hurwitz orbit, then the orbit graph $G$ is the form $(E)$ in Figure \ref{fig:33graph}. The graph $(E)$ is realized as an orbit graph.
\end{prop}

\begin{proof}
If there exists a special triangle in the orbit graph, then as in the proof Lemma \ref{lem:33res}, either $\beta_{2}$ or $\beta_{3}$ is periodic. However, we have shown that in $e(\beta_{2}) \equiv e(\beta_{3}) \equiv \pm 1 \;(mod \;6)$ in Lemma \ref{lem:33res}, this is impossible. Thus, the orbit graph has no special triangles.

So by Lemma \ref{lem:forbidden}, the orbit graph must have a subgraph of the form $(E')$ in Figure \ref{fig:33graph}. Non-existence of special triangles implies that either $a$ or $a'$ (resp. $b$ or $b'$) is a simple vertex. 
If $a$ and $b$ are simple, then we obtain a graph $(E)$.
The graph $(E)$ is realized as the orbit graph of the braid system $(\sigma_{1},\sigma_{2},\sigma_{1})$.
The other cases cannot occur, because it violates the condition in Lemma \ref{lem:33res} {\it 2.} 
\end{proof}

\begin{figure}[htbp]
 \begin{center}
\includegraphics[width=80mm]{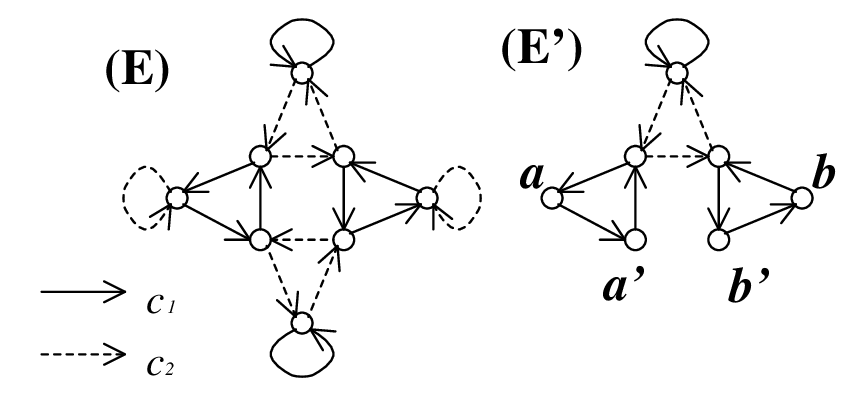}
 \end{center}
 \caption{Orbit graphs of $(3,3)$-periodic systems}
 \label{fig:33graph}
\end{figure}

Now we have classified all orbit graphs of degree $3$, length $3$ irreducible braid systems. Summarizing, we obtain the following result.

\begin{prop}
\label{prop:33bound}
Let $\s$ be an irreducible braid system of degree $3$, length $3$ which has finite Hurwitz orbit. 
Then $\sharp \,( \s \cdot B_{3}) \leq 162 $.
\end{prop}
\begin{proof}
From our list of the orbit graphs, $\sharp\,(\s \cdot F) \leq 9$ holds for all irreducible braid system of degree $3$, length $3$ having finite Hurwitz orbit.
Since $P_{3} = F_{1} \ltimes F_{2} = \langle a_{2,3}\rangle \ltimes F$, $\sharp \,(\s \cdot P_{3}) \leq \sharp( \s \cdot \langle a_{2,3} \rangle) \cdot 9$ holds. Now $\sharp\,(\s \cdot \langle a_{2,3} \rangle )\leq 3$, so we conclude that $\sharp\,(\s \cdot B_{3}) \leq [B_{3}:P_{3}] \cdot 3 \cdot 9 = 162$.
\end{proof}

\subsection{Completion of proof}
Now we complete the proof of Theorem \ref{thm:3finite}.
The last step is to study length $4$ braid systems.

\begin{prop}
\label{prop:4bound}
Let $\s$ be a degree $3$, length $4$ irreducible braid system having finite Hurwitz orbit. Then $\sharp\,(\s \cdot B_{4} )\leq 648$.
\end{prop}
\begin{proof}

From Lemma \ref{lem:restriction}, the possibility of the exponent sum modulo $6$ for irreducible length $4$ braid systems having finite Hurwitz orbit is $(\varepsilon, \varepsilon, \varepsilon, \varepsilon), \varepsilon = \pm 1$. 
Let $\s =(\beta_{1},\beta_{2},\beta_{3},\beta_{4})$. We may assume that $\beta_{3}$ and $\beta_{4}$ do not commute hence $\beta_{3}\beta_{4}$ is periodic. Moreover, since $(\beta_{1},\beta_{2},\beta_{3})$ is a $(3,3)$-periodic system, so by the orbit graph $(E)$ in Figure \ref{fig:33graph}, we may also assume that $\beta_{1}$ and $\beta_{2}$ commute. In particular, $\beta_{1}\beta_{2}$ is non-periodic, and $\beta_{1}\beta_{2}$ does not commute with $\beta_{3}\beta_{4}$.
Assume that the all partial Coxeter elements $C_{\{1,2,3\}}$ of $\s$, $\s \cdot \sigma_{3}^{2}$ and $\s \cdot \sigma_{3}^{4}$ are periodic. That is, $\beta_{1}\beta_{2}\beta_{3}$, $\beta_{1}\beta_{2}\beta_{3}^{(\beta_{3}\beta_{4})}$ and $\beta_{1}\beta_{2}\beta_{3}^{(\beta_{3}\beta_{4})^{2}}$ are periodic.
Then, by Lemma \ref{lem:key}, either $\beta_{3}$ or $\beta_{1}\beta_{2}$ commute with $\beta_{3}\beta_{4}$, which is a contradiction. Hence we may assume that $\beta_{1}\beta_{2}\beta_{3}$ is non-periodic, so $\beta_{1},\beta_{2}$ and $\beta_{3}$ commute.

We consider in $B_{3}'$. Let us put $\beta_{1} \equiv \beta^{p}$, $\beta_{2} \equiv \beta^{q}$, $\beta_{3} \equiv \beta^{r}$ and $\beta_{4} \equiv \gamma$.
Since $\s$ is irreducible, $\gamma$ does not commute with $\beta$. Therefore, all of $\beta_{2}\gamma$, $\beta_{3}\gamma$, $\beta_{2}\beta_{3}\gamma$ are periodic. Then the exponent sum argument shows that their periods are $3,3,2$ respectively. Thus we have an equality $(\beta^{q}\gamma)^{3} \equiv (\beta^{r}\gamma)^{3} \equiv (\beta^{q+r}\gamma)^{2} \equiv 1$.
From this equality, we obtain $\beta^{r}\gamma\beta^{r} \equiv \gamma\beta^{q}\gamma$.
Similar argument for $\beta_{1}$ and $\beta_{3}$ provide an equality $\beta^{r}\gamma\beta^{r} \equiv \gamma\beta^{p}\gamma$, hence we conclude $\beta^{p} \equiv \beta^{q}$. Similarly, by considering $\beta_{1}$ and $\beta_{2}$, we obtain $\beta^{q} \equiv \beta^{r}$. Hence the equality $\beta_{1} \equiv \beta_{2} \equiv \beta_{3}$ holds.

Let $G$ be a subgroup of $B'_{3}$ generated by $\beta^{p}$ and $\gamma$. Then the map  
$\tau: B'_{3}  \rightarrow G$ defined by $\tau([\sigma_{1}])=\beta^{p}$ and $\tau([\sigma_{2}])=\gamma$ is a surjective homomorphism. Now the map $\tau$ induces a surjection between Hurwitz orbits $(\sigma_{1},\sigma_{1},\sigma_{1},\sigma_{2}) \cdot P_{4} $ and $\s \cdot P_{4} $.
Thus, we conclude that $\sharp(\s \cdot B_{4}) \leq 4! \,\sharp (\sigma_{1},\sigma_{1},\sigma_{1},\sigma_{2}) \cdot P_{4}$.
A direct calculation shows $\sharp\, (\sigma_{1},\sigma_{1},\sigma_{1},\sigma_{2}) \cdot P_{4} = 27$, hence we conclude $\sharp\,( \s \cdot B_{4}) \leq 648$.
\end{proof}
 
\begin{rem}
We remark that the upper bound $648$ is achieved by the braid system $(\sigma_{1},\Delta^{2}\sigma_{1}, \Delta^{4} \sigma_{1}, \Delta^{6}\sigma_{2})$.
The above proof implies that the orbit graph of an irreducible braid system of degree $3$, length $4$ with respect to the Hurwitz $P_{4}$-action is obtained as a quotient of the orbit graph of $(\sigma_{1},\sigma_{1},\sigma_{1},\sigma_{2})$. Since the possibilities of such graphs are finite, we can classify the whole patterns of the orbit graphs for $P_{4}$-action. This implies, theoretically we can list all the possibilities of the orbit graphs of finite Hurwitz orbits.  
\end{rem} 
 
\begin{proof}[Proof of theorem \ref{thm:3finite} {\it 2.}]
The assertion {\it 1.} and {\it 3.} are already proved. 
Since we have already studied the irreducible case, we only need to consider the reducible case.
Let $\s=(\beta_{1},\ldots,\beta_{n})$ be a reducible system having finite Hurwitz orbit and $I \coprod J = \{1,2,\ldots n\}$ be the partition appeared in the definition of a reducible system. 

Assume that $\beta_{i}$ and $\beta_{i'}$ do not commute. Then we may assume that $i,i' \in I$. Now for $j \in J$, $\beta_{j}$ commutes with both $\beta_{i}$ and $\beta_{i'}$. Now $\beta_{i}$ and $\beta_{i'}$ does not commute implies that $Z(\beta_{i}) \cap Z(\beta_{i'}) = Z(B_{3})$, so $\beta_{j} \in Z(B_{3})$ for all $j \in J$.
Thus, we have one of 
\begin{enumerate}
\item All the $\beta_{i}$ commute with each other.
\item There exist $i_{1}< i_{2} < \cdots <  i_{k}$ $(2 \leq k \leq 4)$ such that the braid system $(\beta_{i_{1}},\beta_{i_{2}},\ldots,\beta_{i_{k}})$ is irreducible braid system having finite Hurwitz orbit, and $\beta_{j} \in Z(B_{3})$ for $j \not \in \{ i_{1},i_{2},\ldots,i_{k}\}$.
\end{enumerate}

For the first case, we get $\sharp(\s \cdot B_{n}) \leq n!$.
In the second case, we use the inequality of the size of finite Hurwitz orbit for reducible systems we mentioned at Section 1. By Proposition \ref{prop:33bound} and \ref{prop:4bound}, we get $\sharp(\s \cdot B_{n}) \leq 27 \cdot n!$.
\end{proof}

\noindent
\textbf{Acknowledgment.}
The author would like to express his gratitude to Toshitake Kohno for many helpful suggestions. He also wish to thank Yoshiro Yaguchi for drawing the author's attention to the Hurwitz action. This research was supported by JSPS Research Fellowships for Young Scientists.

\vspace{1cm}
\begin{flushright}
\begin{minipage}{.5\textwidth}
\noindent
Tetsuya Ito \\
University of Tokyo\\
3-8-1 Komaba Meguro-Ku, Tokyo \\
Japan\\
e-mail: tetitoh@ms.u-tokyo.ac.jp
\end{minipage}
\end{flushright}

\end{document}